\documentclass[addpoints,12pt,letter]{article}

\usepackage{amscd}
\usepackage{amsfonts}
\usepackage{amsmath}
\usepackage{amssymb}
\usepackage{amsthm}
\usepackage{authblk}
\usepackage{booktabs}
\usepackage{chngpage}
\usepackage{color}
\usepackage{commath}
\usepackage{comment}
\usepackage{diagbox}
\usepackage[shortlabels]{enumitem}
\usepackage{esint}
\usepackage{fge}
\usepackage{float}
\usepackage{graphics}
\usepackage{graphicx}
\usepackage{latexsym}
\usepackage{lineno}
\usepackage{mathrsfs}
\usepackage{mathtools}
\usepackage{multicol}
\usepackage{multirow}
\usepackage{nccmath}
\usepackage{subcaption}
\usepackage{textcomp}
\usepackage[sc]{titlesec}
\usepackage{tikz}
\usepackage{tikz-cd}
\usepackage[utf8x]{inputenc}
\usepackage{xcolor}

\usepackage{lmodern}
\usepackage[T1]{fontenc}

\makeatletter
\newcommand\ReDeclareMathOperator{
\@ifstar{\def\rmo@s{m}\rmo@redeclare}{\def\rmo@s{o}\rmo@redeclare}}
\newcommand\rmo@redeclare[2]{
\begingroup \escapechar\m@ne\xdef\@gtempa{{\string#1}}\endgroup
\expandafter\@ifundefined\@gtempa
{\@latex@error{\noexpand#1undefined}\@ehc}
\relax
\expandafter\rmo@declmathop\rmo@s{#1}{#2}}
\newcommand\rmo@declmathop[3]{
\DeclareRobustCommand{#2}{\qopname\newmcodes@#1{#3}}}
\@onlypreamble\ReDeclareMathOperator
\makeatother

\ReDeclareMathOperator{\AE}{AE}

\DeclareMathOperator{\wind}{wind}

\DeclareSymbolFont{eulargesymbols}{U}{zeuex}{m}{n}
\DeclareMathSymbol{\intop}{\mathop}{eulargesymbols}{"52}
\DeclareMathSymbol{\ointop}{\mathop}{eulargesymbols}{"49}

\newcommand{\al}{\alpha}
\newcommand{\bC}{\mathbb C}
\newcommand{\bD}{\mathbb D}
\newcommand{\be}{\beta}

\newcommand{\bN}{\mathbb N}

\newcommand{\bR}{\mathbb R}

\newcommand{\bT}{\mathbb T}
\newcommand{\bZ}{\mathbb Z}

\newcommand{\cD}{\mathcal D}

\newcommand{\cR}{\mathcal R}

\newcommand{\de}{\delta}
\newcommand{\De}{\Delta}

\newcommand{\e}{\hspace{0.6pt}\mathrm{e}}

\newcommand{\eps}{\varepsilon}

\newcommand{\ga}{\gamma}
\newcommand{\Ga}{\Gamma}

\newcommand{\ka}{\kappa}

\newcommand{\la}{\lambda}
\newcommand{\La}{\Lambda}

\newcommand{\nf}{\infty}
\newcommand{\ol}{\overline}
\newcommand{\om}{\omega}

\newcommand{\ph}{\varphi}

\newcommand{\si}{\sigma}

\newcommand{\tht}{\theta}

\newcommand{\vro}{\varrho}
\newcommand{\x}{\raisebox{0.2mm}{\mbox{\tiny $\times$}\hspace{-0.3mm}}}
\newcommand{\ze}{\zeta}

\renewcommand{\ge}{\geqslant}

\renewcommand{\i}{\hspace{0.2pt}\mathrm{i}}

\renewcommand{\le}{\leqslant}

\renewcommand{\Re}{\mathfrak{\,Re}}

\def\YYint#1#2#3{{\setbox0=\hbox{$#1{#2#3}{\int}$}
\vcenter{\hbox{$#2#3$}}\kern-.52\wd0}}

\usetikzlibrary{arrows,calc,shapes,positioning}
\usetikzlibrary{decorations.markings}
\tikzstyle arrowstyle=[scale=1]
\tikzstyle directed=[postaction={decorate,decoration={markings,
    mark=at position .65 with {\arrow[arrowstyle]{stealth}}}}]
\tikzstyle reverse directed=[postaction={decorate,decoration={markings,
    mark=at position .65 with {\arrowreversed[arrowstyle]{stealth};}}}]

\numberwithin{equation}{section}

\newtheorem{lemma}{Lemma}[section]
\newtheorem{theorem}[lemma]{Theorem}

\newtheorem{conjecture}[lemma]{Conjecture}

\theoremstyle{definition}

\newtheorem{remark}[lemma]{Remark}

\title{\textsc{Fast non-Hermitian Toeplitz eigenvalue computations, joining matrix-less algorithms and FDE approximation matrices}}

\author[1]{M. Bogoya\thanks{johan.bogoya@correounivalle.edu.co}}
\author[2]{S.M. Grudsky\thanks{grudsky@math.cinvestav.mx}}
\author[3]{S. Serra--Capizzano\thanks{s.serracapizzano@uninsubria.it}}

\affil[1]{\footnotesize Universidad del Valle, Cali, Colombia.\\
Departamento de Matem\'aticas.}
\affil[2]{\footnotesize CINVESTAV, Mexico, Mexico D.F.\\
Departamento de Matem\'aticas.}
\affil[3]{\footnotesize University of Insubria, Como, Italy.\\
Dipartimento di Scienza e Alta Tecnologia.}

\date{\small\today}

\begin{document}

\maketitle

\begin{abstract}
The present work is devoted to the eigenvalue asymptotic expansion of the Toeplitz matrix $T_{n}(a)$ whose generating function $a$ is complex valued and has a power singularity at one point. As a consequence, $T_{n}(a)$ is non-Hermitian and we know that the eigenvalue computation is a non-trivial task in the non-Hermitian setting for large sizes. We follow the work of Bogoya, Böttcher, Grudsky, and Maximenko and deduce a complete asymptotic expansion for the eigenvalues. After that, we apply matrix-less algorithms, in the spirit of the work by Ekström, Furci, Garoni, Serra-Capizzano et al, for computing those eigenvalues. Since the inner and extreme eigenvalues have different asymptotic behaviors, we worked on them independently, and combined the results to produce a high precision global numerical and matrix-less algorithm.

The numerical results are very precise and the computational cost of the proposed algorithms is independent of the size of the considered matrices for each eigenvalue, which implies a linear cost when all the spectrum is computed. From the viewpoint of real world applications, we emphasize that the matrix class under consideration includes the matrices stemming from the numerical approximation of fractional diffusion equations. In the final conclusion section a concise discussion on the matter and few open problems are presented.
\medskip\\
\noindent\textbf{MSC Classes}: Primary 15B05, 65F15, 47B35. Secondary 15A18, 47A38.
\end{abstract}

\section{Introduction}

Let $\bT$ be the complex unit circle. For a complex-valued function $a\in L^{1}(\bT)$ let $a_{k}$ be its $k$th Fourier coefficient, that is
\[a_{k}\equiv\frac{1}{2\pi}\int_{0}^{2\pi}a(\e^{\i\tht})\e^{-\i k\tht}\dif\tht,\quad\ k\in\bZ, \ \i^2=-1.\]
The $n\times n$ Toeplitz matrix with {\em generating function} $a$ is given by $T_{n}(a)\equiv(a_{j-k})_{j,k=0}^{n-1}$, and is characterized by having the fixed constant $a_k$ along the $k$th diagonal, parallel to the main one. In formulae, we have
\begin{equation*}
T_{n}(a)=\begin{bmatrix}
a_0 & a_{-1} & a_{-2} & \cdots & \cdots & a_{-(n-1)} \\
a_{1} & \ddots & \ddots & \ddots & & \vdots\\
a_{2} & \ddots & \ddots & \ddots & \ddots & \vdots\\
\vdots & \ddots & \ddots & \ddots & \ddots & a_{-2}\\
\vdots & & \ddots & \ddots & \ddots & a_{-1}\\
a_{n-1} & \cdots & \cdots & a_{2} & a_{1} & a_0
\end{bmatrix}.
\end{equation*}

We will follow the notation of \cite{BoBo12}. Let $H^{\nf}$ be the Hardy space of (boundary values of) bounded analytic functions in the complex unit disk $\bD$, $\bT=\partial \bD$. For a continuous generating function $a$, we denote by
\[
\wind_{\la}(a)\equiv\frac{1}{2\pi\i}\oint_{\cR(a)}\frac{\dif \ze}{\ze-\la},
\]
the winding number of $a$ in connection with a point $\la\in\bC\setminus\cR(a)$ where $\cR(a)$ is the range of $a$. Intuitively, $\wind_{\la}(a)$ is the number of times that the range of $a$ travels around the point $\la$ in counterclockwise sense. We also denote by $\cD(a)$ the set of all $\la\in\bC$ for which $\wind_{\la}(a)\ne0$. In plane words $\cD(a)$ is the collection of all the points trapped by $\cR(a)$. In the first part of the current work, we will study the eigenvalues of $T_{n}(a)$ for functions $a(t)=\frac{1}{t}(1-t)^{\al}f(t)$ satisfying the following properties:
\begin{enumerate}[(i)]
\item The constant $\al$ belongs to the interval $(0,1)$ and $f\in C^{\nf}(\bT)\cap H^{\nf}$.
\item $f(z)\ne0$ for $z\in\ol\bD$, and $f$ has an analytic extension to an open neighborhood $W$ of $\bT\setminus\{1\}$ not containing the point $1$.
\item The range of $a$ is a Jordan curve in $\bC$ which is the boundary of a connected region in the complex plane.
\end{enumerate}

Strictly speaking, the complex-valued function $(1-z)^{\al}$ is multi-valued, hence we take the argument specified by $\arg(1-z)^{\al}\in(-\al\pi,\al\pi]$ when $\arg(1-z)\in(-\pi,\pi]$. Property~(ii) is necessary because, in that case, $b$ belongs to the Hardy space $H^{\nf}$ and the nice relation $T_{n}^{-1}(b)=T_{n}\big(\frac{1}{b}\big)$ holds, which is used for expressing the determinant $D_{n}(a)$ of $T_{n}(a)$ as a Fourier integral. Additionally, the condition $\al\in(0,1)$ implies that $a$ has a power singularity at $t=1$, i.e. no continuous derivatives there.

At this point it is worth noticing that the considered set of functions includes the generating functions of a class of Hessenberg Toeplitz matrices stemming from the numerical approximation of certain fractional diffusion equations (FDEs); see \cite{BoGr22b,BoGr22a,DoMa16}. In the latter context, the parameter $\al$ belongs to the open interval $(0,2)$. We remind that the quoted FDEs with noninteger parameter $\al \in (0,2)$ have gained a tremendous attention in real world applications, because they model anomalous diffusion processes arising in Biology, Physics, etc. (see again \cite{DoMa16} and references therein).

Property (ii) also guaranties that the only point of $\bT$ where $a$ is not analytic is $t=1$, and finally, property (iii) amounts in avoiding that $\cR(a)$ shows the presence of loops. The resulting Toeplitz matrices $T_{n}(a)$ are lower Hessenberg, meaning that they can be obtained by adding one upper diagonal to a lower triangular matrix.

We emphasize that this knowledge is important when analyzing fast methods for the solution of large linear systems arising from the numerical approximation of the considered FDEs, via preconditioned Krylov methods, multigrid algorithms, and multi-iterative solvers (see e.g. \cite{DoGa15} and references therein).

In \cite{BoBo12} an asymptotic expansion was deduced for the eigenvalues bounded away from zero, which the authors named {\em inner}, while in \cite{BoGr18} the {\em extreme} (closest to zero) eigenvalues were studied. With a different technique, the authors were able to obtain an asymptotic expression for the eigenvalues located near to the singularity point. The result is reminiscent of the works by S. Parter \cite{Pa61a,Pa61b} and H. Widom \cite{Wi58}, for the real-valued case; see also \cite{BoGr98,Se98b,SeTi99,Ti03} for a more complete picture on the subject, including newer matrix theoretic tools.

The case $\al\in(1,2)$ or more generally, $\al\in\bR_{+}\setminus\bZ$, can be managed analogously for the inner eigenvalues, while the extreme eigenvalues remain a challenge that will be the subject of a future investigation. For this reason we restrict ourselves to the case $\al\in(0,1)$.

Regarding computational costs, the best option relies in employing extrapolation techniques based on the asymptotic expansion, because the latter leads to a direct eigenvalue approximation which is independent of the matrix entries or its size. Indeed the computation can be done in two different ways: the first one is to use the analytic expression of the coefficient functions, if available; the second one is to employ a numerical and matrix-less algorithm to approximate them. After the respective precomputing phases, in both cases, the calculation will need only a fraction of a second in a standard computer. The standard option is to use any commercial eigensolver which are limited by the matrix size and its conditioning. To put it in numbers, the eigenvalue calculation of a $4000$ size moderately conditioned matrix with 50 precision digits, can be done in a week by a standard computer. If the matrix size is $10000$, only a supercomputer can handle the same task, but if the matrix size is $10^{10}$ (a size appearing in several applications) no computer can do it.

Taking inspiration from the literature, in the present work we split our study into the mentioned cases. The article is organized as follows: in Section~\ref{sc:inner eigs} we deduce a complete asymptotic expansion for the inner eigenvalues of $T_{n}(a)$ which served as the theoretical ground for the subsequent matrix-less numerical algorithm. In Section~\ref{sc:extreme eigs} we use a previous theoretical result to conjecture an expansion which was the base for the respective matrix-less numerical algorithm. In Section~\ref{sc:num} we join the two processes in order to design a global algorithm approximating all the eigenvalues with high accuracy: in this section we also report and discuss our numerical experiments. Finally, in Section~\ref{sc:end} we draw some future works.

\section{Inner eigenvalues}\label{sc:inner eigs}

In the paper \cite{BoBo12}, the authors employed the Baxter--Schmidt formula (see \cite[\S2.3]{BoGr05}) for Toeplitz determinants to express $D_{n}(a-\la)$ as a Fourier integral. In particular they also showed that if $W_{0}$ is a small open neighborhood of zero in $\bC$, for every point $\la\in\cD(a)\cap(a(W)\setminus W_{0})$, there exists a unique point $t_{\la}\in\bC$, with $|t_{\la}|>1$ such that $a(t_{\la})=\la$. See Figure~\ref{fg:Neigh}.

In order to attain an asymptotic expansion for the eigenvalues of $T_{n}(a)$, we are based on \cite{BoBo12} where the authors obtained an asymptotic expression involving only two terms, our aim is to obtain an expression with any number of terms.
The following is our main result.

\begin{figure}[H]
\centering
\includegraphics[width=0.9\textwidth]{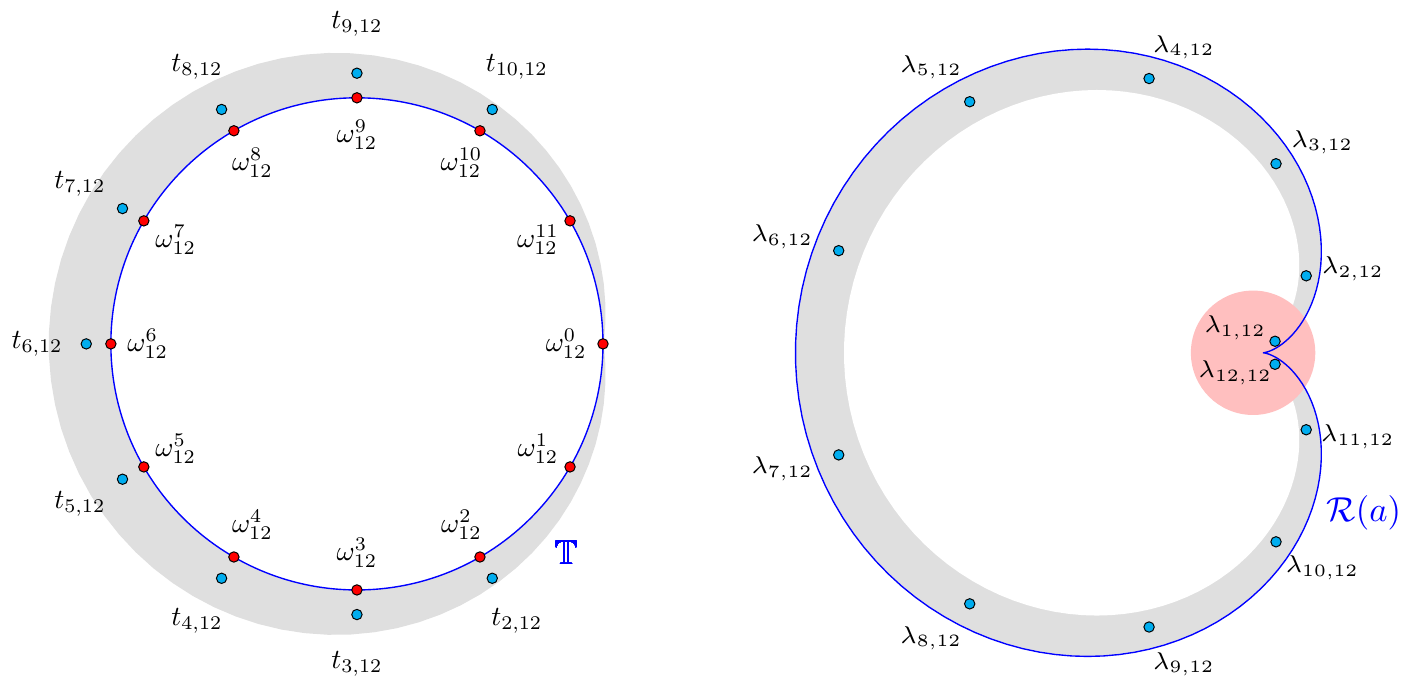}
\caption{The range and neighborhoods of the generating function $a(t)=\frac{1}{t}(1-t)^{\al}$ with $\al=\frac{3}{4}$. Left: For $n=12$ the red dots are the roots of unity $\om_{n}^{j}$, the cyan dots are the sequence $(t_{j,n})$ described in Theorem~\ref{th:InnerExp}, that is $\la_{j}(T_{n}(a))=a(t_{j,n})$, $j=0,\ldots,11$. The gray shaded region is the neighborhood $W\setminus\ol{\bD}$ where $a$ has an analytic extension and where the sequence $t_{j,n}$ is located. Right: For $n=12$ the cyan dots are the eigenvalues $\la_{j}(T_{12}(a))=\la_{j,12}$, $j=1,\ldots,12$, the gray shaded region is the neighborhood $a(W\setminus\ol{\bD})$ where the inner eigenvalues are located. The light red shaded region is the neighborhood $W_{0}=B(0,\eps)$, where the extreme eigenvalues are located.}\label{fg:Neigh}
\end{figure}

\begin{theorem}\label{th:InnerExp}
Let $a(t)=\frac{1}{t}b(t)$ be a function satisfying the properties \textup{(i)--(iii)}. Then for every small open neighborhood $W_{0}$ of zero in $\bC$ and every index $j$ with $a(\om_{n}^{j})\notin W_{0}$, $t_{j,n}\equiv t_{\la_{j,n}}$ admits the following asymptotic expansion
\begin{eqnarray*}
t_{j,n}&\sim&\om_{n}^{j}n^{(\al+1)\frac{1}{n}}\bigg\{1+\sum_{r=0}^{\nf} \sum_{m=1}^{\nf}\sum_{\ell=0}^{m-1}p_{r,m,\ell}(\om_{n}^{j})\frac{\log^{\ell}(n)}{n^{\al r+m}}\bigg\}\\
&=&\om_{n}^{j}\bigg\{1+\sum_{r=0}^{\nf}\sum_{m=1}^{\nf}\sum_{\ell=0}^{m-1} q_{r,m,\ell}(\om_{n}^{j})\frac{\log^{\ell}(n)}{n^{\al r+m}}\bigg\},
\end{eqnarray*}
where
\begin{itemize}
\item $\la_{j}(T_{n}(a))=a(t_{j,n})$ and $\om_{n}\equiv\exp(-\frac{2\pi\i}{n})$;
\item the coefficients $p_{r,m,\ell}$ and $q_{r,m,\ell}$ are functions in $C(\bT)$ depending only on the generating function $a$.
\end{itemize}
\end{theorem}

To help the visualization of the expansion in the previous theorem, take $h\equiv\frac{1}{n}$ and arrange the involved powers of $h$ with the following array:

\medskip
\begin{center}
\begin{tabular}{lllllll}
$h$ & & & $h^{\al+1}$ & & & $\cdots$\\[1ex]
$h^{2}\log(n)$ & $h^{2}$ & & $h^{\al+2}\log(n)$ & $h^{\al+2}$ & & $\cdots$\\[1ex]
$h^{3}\log^{2}(n)$ & $h^{3}\log(n)$ & $h^{3}$ & $h^{\al+3}\log^{2}(n)$ & $h^{\al+3}\log(n)$ & $h^{\al+3}$ & $\cdots$\\[1ex]
$\vdots$ & & & $\vdots$ & & & $\ddots$\\
\end{tabular}
\end{center}
\medskip
 
\noindent Hence the terms can be ordered, but the ordering depends on $\al$. In \cite[Th.\,1.2]{BoBo12} the coefficient $p_{0,1,0}$ was exactly calculated as
\[p_{0,1,0}(z)=\,\log\Big(\frac{a^{2}(z)}{c\,a'(z) z^{2}}\Big)
\quad\mbox{with}\quad c=\frac{1}{\pi}f(1)\Ga(\al+1)\sin(\al\pi),\]
where $\Ga$ is the well-known Gamma function. In this section, instead of the exact calculation of the remaining coefficients $p_{r,m,\ell}$ or $q_{r,m,\ell}$, we are interested in justifying the whole expansion and then, use it for proposing a matrix-less numerical algorithm, in the same fashion as in  \cite{BoEk22a,BoSe22a,BoSe22c,EkFu18b,EkGa19,EkGa18}.
Indeed, we remind that the quoted procedures are all essentially of extrapolation type and hence the formal (or numerical) expression of the coefficients is not required, likewise the extrapolation techniques for the computation of definite integrals in the sense of Romberg \cite[\S3.4]{StBu10}.

In asymptotic analysis we say that a function $f$ admits the asymptotic expansion $\sum_{k=0}^{\nf}a_{k}(x-x_{0})^{k}$ at $x_{0}$, denoted as
\[f(x)\sim\sum_{k=0}^{\nf}a_{k}(x-x_{0})^{k},\]
if we have $f(x)=\sum_{k=0}^{m-1}a_{k}(x-x_{0})^{k}+O(|x-x_{0}|^{m})$ for every $m\in\bZ_{+}$. For proving Theorem~\ref{th:InnerExp} we will need some preliminary findings. We start with the following well-known result, which is, for example in \cite[p.\,97]{Fe77}. We present it here for reader convenience and without a proof.

\begin{theorem}\label{th:Fedoryuk}
Let $\ga>0$, $\de>0$, $u\colon[0,\de]\to\bC$ be a $C^{\nf}$ function with $u^{(m)}(\de)=0$ for all $m\ge0$. Then, as $n\to\nf$
\[\int_{0}^{\de}\tht^{\ga-1}u(\tht)\e^{\i n\tht}\dif\tht\sim\sum_{m=0}^{\nf} \frac{u^{(m)}(0)\Ga(m+\ga)\e^{\i\frac{\pi}{2}(m+\ga)}}{m!\,n^{m+\ga}}.\]
\end{theorem}

For a generating function $a(t)=\frac{1}{t}b(t)$ satisfying the properties (i)--(iii), we can write
\[b(\e^{\i\tht})=(1-\e^{\i\tht})^{\al}f(\e^{\i\tht})=\e^{-\frac{\i}{2}\al\pi}\tht^{\al} f_{0}(\tht)f(\e^{\i\tht}),\]
where $f_{0}$ belongs to the class $C^{\nf}[-\de,\de]$. Hence $b$ has a zero of order $\al$ at $\tht=0$.
Let $\de$ be a small positive constant, $\phi\colon\bR\to[0,1]$ a $C^{\nf}$ function with support in $[-\de,\de]$, and $\phi|_{[-\eps,\eps]}\equiv1$ for some $0<\eps<\de$. For each $r\in\bZ_{+}$ consider the auxiliary function
\begin{equation}\label{eq:urDef}
u_{r}(\tht)\equiv f_{0}^{r}(-\tht)f^{r}(\e^{-\i\tht})\phi(\tht)\e^{\i\tht(r+1)},
\end{equation}
and let $D_{n}(a)$ be the determinant of the Toeplitz matrix $T_{n}(a)$. The following result gives us an expansion for the determinant of $T_{n}(a-\la)$, which will be the starting point for obtaining the expansion of the eigenvalues.

\begin{theorem}\label{th:DetExp}
Let $a(t)=\frac{1}{t}b(t)$ be a function satisfying the properties \textup{(i)--(iii)}. Then the determinant of $T_{n}(a-\la)$ is given by
\[D_{n}(a-\la)=(-1)^{n}b_{0}^{n+1}\be_{n}(\la),\]
where $\be_{n}$ admits the asymptotic expansion
\[\be_{n}(\la)\sim\frac{-1}{t_{\la}^{n+2} a'(t_{\la})}-\frac{1}{\pi}
\sum_{r,m=1}^{\nf}\frac{\Ga(\al r+m)\Re\{u_{r}^{(m-1)}(0)\e^{\frac{\i}{2}\pi(2\al r+m)}\}}{(m-1)!a^{r+1}(t_{\la})n^{\al r+m}},\]
$b_{0}\ne0$ is the zeroth Fourier coefficient of $b$, and $u_{r}$ is given by \eqref{eq:urDef}.
\end{theorem}

\begin{proof}
Recall that $\cD(a)$ is the set of all complex points $z$ with $\wind_{\la}(z)\ne0$.
From \cite[Lm.\,4.1 and Lm.\,4.2]{BoBo12}, for any $\la\in\cD(a)$, we reach $D_{n}(a-\la)=(-1)^{n}b_{0}^{n+1} \be_{n}(\la)$, where $b_{0}\ne0$ is the zeroth Fourier coefficient of $b$ and $\be_{n}$ is a linear combination of Fourier-type integrals, admitting the representation
\begin{equation}\label{eq:Det}
\be_{n}(\la)=\frac{-1}{t_{\la}^{n+2} a'(t_{\la})}-\frac{1}{2\pi}\sum_{r=1}^{\nf}\frac{I(r,n)}{a^{r+1}(t_{\la})}+O(n^{-\nf}),
\end{equation}
the order term is uniform in $\la$ and
\[I(r,n)\equiv\int_{-\de}^{\de}b^{r}(\e^{\i\tht})\phi(\tht)\e^{-\i\tht(r+1)}\e^{-\i n\tht}\dif\tht.\]
We will expand the integral $I$ with the aid of Theorem~\ref{th:Fedoryuk}. Note that
\begin{eqnarray*}
I(r,n)&=&\e^{-\frac{\i}{2}\al r\pi} \int_{-\de}^{\de} \tht^{\al r} u_{r}(-\tht)\e^{-\i n\tht}\dif\tht\\
&=&\e^{-\frac{\i}{2}\al r\pi} \bigg\{\ol{\int_{0}^{\de} \tht^{\al r}\ol{u_{r}(-\tht)}\e^{\i n\tht}\dif\tht}+\e^{\i \al\pi r}\int_{0}^{\de} \tht^{\al r}u_{r}(\tht)\e^{\i n\tht}\dif\tht\bigg\},
\end{eqnarray*}
where $u_{r}$ is given by \eqref{eq:urDef}. Using the relation $\ol{u_{r}(\tht)}=u_{r}(-\tht)$ and applying Theorem~\ref{th:Fedoryuk} twice, we obtain
\begin{eqnarray*}
I(r,n)&\sim&\e^{-\frac{\i}{2}\al r\pi} \bigg\{\ol{\sum_{m=1}^{\nf}\frac{\Ga(\al r+m)}{(m-1)!}u_{r}^{(m-1)}(0)\e^{\frac{\i}{2}\pi(\al r+m)}n^{-\al r-m}}\\
&&+\e^{\i\al\pi r}\sum_{m=1}^{\nf} \frac{\Ga(\al r+m)}{(m-1)!}u_{r}^{(m-1)}(0)\e^{\frac{\i}{2}\pi(\al r+m)}n^{-\al r-m}\bigg\}\\
&=&2\sum_{m=1}^{\nf} \frac{\Ga(\al r+m)}{(m-1)!}\Re\{u_{r}^{(m-1)}(0)\e^{\frac{\i}{2}\pi(2\al r+m)}\}n^{-\al r-m}.
\end{eqnarray*}
Finally, replacing in \eqref{eq:Det}, we attain
\[\be_{n}(\la)\sim\frac{-1}{t_{\la}^{n+2} a'(t_{\la})}-\frac{1}{\pi}
\sum_{r,m=1}^{\nf}\frac{\Ga(\al r+m)}{(m-1)!a^{r+1}(t_{\la})}\Re\{u_{r}^{(m-1)}(0)\e^{\frac{\i}{2}\pi(2\al r+m)}\}n^{-\al r-m},\]
which proves the assertion of the theorem.
\end{proof}

The next step consists in solving the relation $D_{n}(a-\la)=0$ for $t_{\la}$, which, with the aid of Theorem~\ref{th:DetExp}, leads to
\[t_{\la}\sim\bigg\{\frac{t_{\la}^{2}a'(t_{\la})}{a^{2}(t_{\la})}
\sum_{r,m=0}^{\nf}\frac{\vro(r,m)}{a^{r}(t_{\la})n^{\al (r+1)+(m+1)}}\bigg\}^{-\frac{1}{n}},\]
where
\[\vro(r,m)\equiv-\frac{\Ga(\al(r+1)+(m+1))}{\pi m!}\Re\{u_{r}^{(m)}(0)\e^{\frac{\i}{2}\pi(2\al(r+1)+(m+1))}\}.\]
In order to expand the above expression and determine the powers of $n$ involved, we note that $\vro(0,0)=\frac{1}{\pi}\Ga(\al+1)\sin(\al\pi)f(1)\ne0$ and consequently we write
\begin{eqnarray}\label{eq:tlaExp2}
t_{\la}&\sim&\om_{n}^{j}\bigg\{\frac{\vro(0,0)t_{\la}^{2}a'(t_{\la})}{a^{2}(t_{\la})n^{\al+1}}\bigg\}^{-\frac{1}{n}}\bigg\{1+\frac{1}{\vro(0,0)}\hspace{-3mm}
\sum_{\substack{r,m=0\\(r,m)\ne(0,0)}}^{\nf}\frac{\vro(r,m)}{a^{r}(t_{\la})n^{\al r+m}}\bigg\}^{-\frac{1}{n}}\notag\\
&=&\om_{n}^{j}n^{(\al+1)\frac{1}{n}}\bigg\{\frac{a^{2}(t_{\la})}{\vro(0,0)t_{\la}^{2}a'(t_{\la})}\bigg\}^{\frac{1}{n}}
\bigg\{1+\frac{1}{\vro(0,0)}\hspace{-3mm}\sum_{\substack{r,m=0\\(r,m)\ne(0,0)}}^{\nf}\frac{\vro(r,m)}{a^{r}(t_{\la})n^{\al r+m}}\bigg\}^{-\frac{1}{n}},
\end{eqnarray}
where $j=0,\ldots,n-1$ and $\om_{n}=\exp(-\frac{2\pi\i}{n})$. Now we want to expand the previous expression until the outer exponents $\pm \frac{1}{n}$ disappear. The main aim here is to determine the powers of $n$ involved instead of the value of the respective exponents. One possible strategy for this aim is as follows. Take $h=\frac{1}{n}$. The expansion of the first term involves $h$ with all natural exponents, i.e. $h,h^{2},h^{3},h^{4},\ldots$, while the expansion of the second term involves all the exponents in the array
\[M\equiv\begin{bmatrix}
 & h^{\al} & h^{2\al} & h^{3\al} & \cdots\\
 h & h^{\al+1} & h^{2\al+1} & h^{3\al+1} & \cdots\\
 h^{2} & h^{\al+2} & h^{2\al+2} & h^{3\al+2} & \cdots\\
 h^{3} & h^{\al+3} & h^{2\al+3} & h^{3\al+3} & \cdots\\
 \vdots &\vdots & \vdots & \vdots & \ddots
 \end{bmatrix},\]
combined as $1+hM+h(h+1)M^{2}+h(h+1)(h+2)M^{3}+\ldots$. Note that the entire powers of the array $M$ will involve powers of $h$ following the next pattern
\[M^{2}=\begin{bmatrix}
 & & h^{2\al} & h^{3\al} & \cdots\\
  & h^{\al+1} & h^{2\al+1} & h^{3\al+1} & \cdots\\
 h^{2} & h^{\al+2} & h^{2\al+2} & h^{3\al+2} & \cdots\\
 h^{3} & h^{\al+3} & h^{2\al+3} & h^{3\al+3} & \cdots\\
 \vdots &\vdots & \vdots & \vdots & \ddots
 \end{bmatrix},\ 
 M^{3}=\begin{bmatrix}
 & & & h^{3\al} & \cdots\\
  & & h^{2\al+1} & h^{3\al+1} & \cdots\\
 & h^{\al+2} & h^{2\al+2} & h^{3\al+2} & \cdots\\
 h^{3} & h^{\al+3} & h^{2\al+3} & h^{3\al+3} & \cdots\\
 \vdots &\vdots & \vdots & \vdots & \ddots
 \end{bmatrix}.\]
Hence, after expanding, multiplying, and ordering, we seize that the second factor in \eqref{eq:tlaExp2} involves the constant $1$ plus terms involving $h$ with all the powers in the set $\{\al r+m\colon r,m\in\bZ,\,r\ge0,\,m\ge1\}$.
Finally, combining the expansions of the two factors of \eqref{eq:tlaExp2} we obtain
\begin{equation}\label{eq:Rec}
t_{\la}\sim\om_{n}^{j}n^{(\al+1)\frac{1}{n}}\bigg\{1+\sum_{\substack{r\ge0\\m\ge1}} \frac{\xi_{r,m}(t_{\la})}{n^{\al r+m}}\bigg\},
\end{equation}
where $\xi_{r,m}$ are functions in $C^{\nf}(W\setminus\ol{\bD})$, that can be determined explicitly. The previous relation is an implicit equation for $t_{\la}$. In \cite[Th.\,5.1]{BoBo12} it was proved that there exists a collection of mutually disjoint neighborhoods of each $\om_{n}^{j}$, $j=0,\ldots,n-1$, containing exactly one solution of \eqref{eq:Rec}. We are now ready to prove our main result.

\begin{proof}[Proof of Theorem~\ref{th:InnerExp}]
Let $t_{\la}$ be a solution of \eqref{eq:Rec}. Since $a(t_{\la})=\la$ must be an eigenvalue of $T_{n}(a)$, we conclude that there are exactly $n$ solutions and we denote them as $t_{j,n}\equiv t_{\la_{j,n}}$, $j=1,\ldots,n$. Now we iterate over \eqref{eq:Rec}. After expanding the term $n^{(\al+1)\frac{1}{n}}$ and multiplying, we reach
\begin{equation}\label{eq:RecF}
t_{\la}\sim\om_{n}^{j}\bigg\{1+\sum_{\ell=1}^{\nf}\ze^{(1)}_{l}(t_{\la})\frac{\log^{\ell}(n)}{n^{\ell}}+\sum_{r=0}^{\nf}\sum_{m=1}^{\nf}\sum_{\ell=0}^{m-1} \ze^{(2)}_{r,m,l}(t_{\la})\frac{\log^{\ell}(n)}{n^{\al r+m}}\bigg\},
\end{equation}
where the functions $\ze_{l}^{(1)}$ and $\ze_{r,m.l}^{(2)}$ belong to $C^{\nf}(W\setminus\ol{\bD})$ and can be determined explicitly. Since the expanded terms in \eqref{eq:RecF} have order $\frac{\log(n)}{n}$, our first iteration is $t_{\la}^{(1)}=\om_{n}^{j}+O(\frac{\log(n)}{n})$. The second iteration is
\begin{eqnarray*}
t_{\la}^{(2)}&\sim&\om_{n}^{j}n^{(\al+1)\frac{1}{n}}\bigg\{1+\sum_{\substack{r\ge0\\m\ge1}}^{\nf} \frac{\xi_{r,m}(t_{\la}^{(1)})}{n^{\al r+m}}\bigg\}\\
&=&\om_{n}^{j}n^{(\al+1)\frac{1}{n}}\bigg\{1+\sum_{r=0}^{\lfloor\frac{1}{\al}\rfloor}\frac{\xi_{r,1}(\om_{n}^{j})}{n^{\al r+1}}+O\Big(\frac{\log(n)}{n^{2}}\Big)\bigg\}\\
&=&\om_{n}^{j}\bigg\{1+\sum_{r=0}^{\lfloor\frac{1}{\al}\rfloor}\frac{\xi_{r,1}(\om_{n}^{j})}{n^{\al r+1}}+O\Big(\frac{\log(n)}{n^{2}}\Big)\bigg\}.
\end{eqnarray*}
To reveal the shape in the full expansion we need a third iteration,
\begin{eqnarray*}
t_{\la}^{(3)}&\sim&\om_{n}^{j}n^{(\al+1)\frac{1}{n}}\bigg\{1+\sum_{\substack{r\ge0\\m\ge1}}^{\nf} \frac{\xi_{r,m}(t_{\la}^{(2)})}{n^{\al r+m}}\bigg\}\\
&=&\om_{n}^{j}n^{(\al+1)\frac{1}{n}}\bigg\{1+\sum_{r=0}^{\lfloor\frac{2}{\al}\rfloor}\frac{\xi_{r,1}(\om_{n}^{j})}{n^{\al r+1}}+\sum_{r=0}^{\lfloor\frac{1}{\al}\rfloor} \sum_{\ell=0}^{1}\ze_{l,r}^{(3)}(\om_{n}^{j})\frac{\log^{\ell}(n)}{n^{\al r+2}}+O\Big(\frac{\log(n)}{n^{3}}\Big)\bigg\}\\
&=&\om_{n}^{j}\bigg\{1+\sum_{r=0}^{\lfloor\frac{2}{\al}\rfloor}\frac{\ze_{r}^{(4)}(\om_{n}^{j})}{n^{\al r+1}}+\sum_{r=0}^{\lfloor\frac{1}{\al}\rfloor} \sum_{\ell=0}^{1}\ze_{j,r}^{(5)}(\om_{n}^{j})\frac{\log^{\ell}(n)}{n^{\al r+2}}+O\Big(\frac{\log(n)}{n^{3}}\Big)\bigg\},
\end{eqnarray*}
where the functions $\ze_{l,r}^{(3)}$, $\ze_{r}^{(4)}$, and $\ze_{j,r}^{(5)}$ belong to $C^{\nf}(W\setminus\ol{\bD})$.
By continuing this process we finish the theorem.
\end{proof}

\subsection{Algorithm}\label{ssc:InnAlg}

As in previous works, one key observation is that $\om_{n}^{j}$ can be set unchanged under several combinations of $n$ and $j$, for example
\[\om_{n}^{j}=\om_{cn}^{cj},\]
for any constant $c$, thus we take $n_{\ell}\equiv \ell n_{1}$ and $j_{\ell}\equiv\ell j_{1}$ for $\ell\in\bN$. The index $j_{\ell}$ depends on $j_{1}$, and similarly, the matrix sizes $n_{\ell}$ depend on $n_{1}$, but for notation simplicity, we suppressed those dependencies. Therefore we have
\[\si_{j_{1}}\equiv\om_{n_{1}}^{j_{1}}=\om_{n_{2}}^{j_{2}}=\cdots=\om_{n_{k-1}}^{j_{k-1}},\]
see Figure~\ref{fg:Grid}.
The sequence $\{\om_{n}^{j}\}_{j=0}^{n-1}$ corresponds to the $n$ roots of the unity in the complex plane $\bC$ and is a regular partition of the unit circle $\bT$ with step angle of $\frac{2\pi}{n}$. We assume that $n_{1}$ and $k$ are fixed natural numbers and $n\gg n_{1}$.

Our algorithm is composed by three phases: precomputing, extrapolation, and interpolation. We now describe them in detail.

\begin{figure}[H]
\centering
\includegraphics[width=0.5\textwidth]{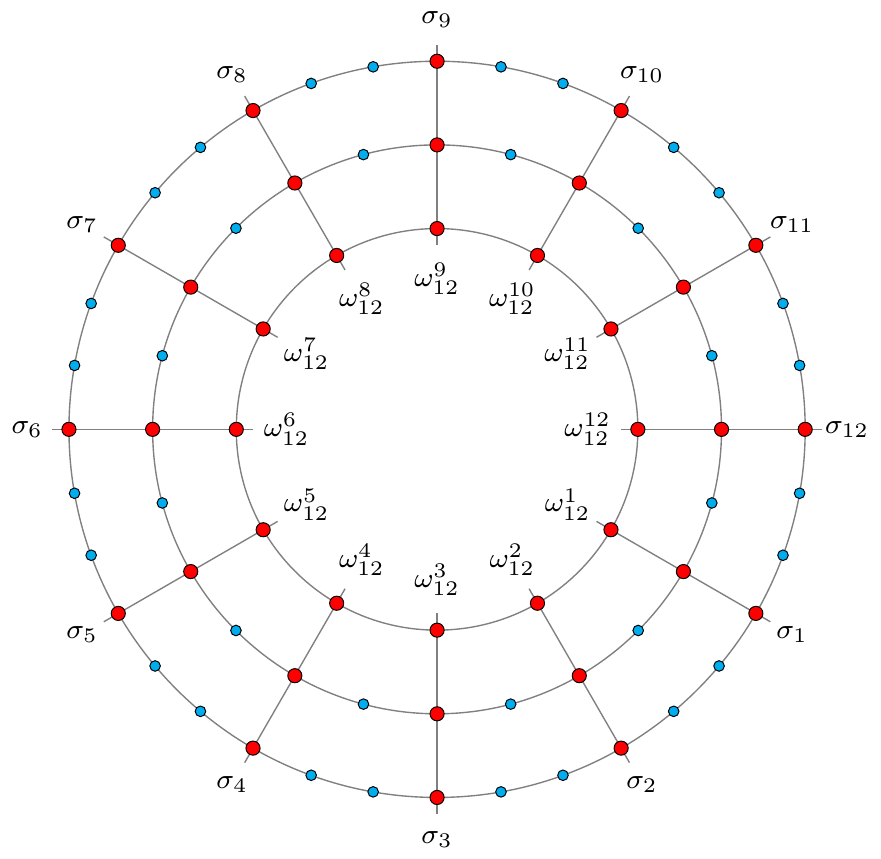}
\caption{For $n_{1}=12$ the regular circular grids $\{\om_{n_{\ell}}^{j}\}$ for $j=1,\ldots,n_{\ell}$ and $\ell=1,2,3$. In the precomputing phase we will need to calculate the eigenvalues of $T_{n_{\ell}}(a)$ for $\ell=1,\ldots,k-1$ corresponding to the blue and red dots combined, but in the interpolation phase, we will use only the eigenvalues corresponding to the red dots, that is $\la_{j_{\ell}}(T_{n_{\ell}}(a))$ for $j_{1}=1,\ldots,n_{1}$ and $\ell=1,\ldots,k-1$.}\label{fg:Grid}
\end{figure}

\noindent{\textbf{Precomputing.}
In this phase we need to calculate the eigenvalues of $T_{n}(a)$ for $n=n_{1},\ldots,n_{k-1}$. The numerical computations for the ``exact'' eigenvalues are conducted with the \texttt{BigFloat} data type in \textsc{Julia v.1.7.2} using 1024 bit precision which approximately corresponds to 300 decimal digits of accuracy. We use an AMD Epyc with 32 cores/64 threads and 1024 GB RAM. The remaining computations are carried using \textsc{Mathematica v13}.
\newpage

\noindent{\textbf{Extrapolation.} Here we need to simplify the expansion in Theorem \ref{th:InnerExp}, for that purpose, let $\xi_1(n),\xi_2(n),\ldots$, be the asymptotic ordering of the set
\[\Big\{\frac{\log^{\ell}(n)}{n^{\al r+m}}\colon r\ge0,\ m\ge1,\ \ell=0,\ldots,m-1\Big\}.\]
This ordering depends on $\al$, for example for $\al=\frac{3}{4}$ we obtain
\begin{align*}
\xi_{1}(n)&=\frac{1}{n}, & \xi_{2}(n)&=\frac{1}{n^{\al+1}}, & \xi_{3}(n)&=\frac{\log(n)}{n^{2}}, & \xi_{4}(n)&=\frac{1}{n^{2}}\\
\xi_{5}(n)&=\frac{1}{n^{2\al+1}}, & \xi_{6}(n)&=\frac{\log(n)}{n^{\al+2}}, & \xi_{7}(n)&=\frac{1}{n^{\al+2}}, & \xi_{8}(n)&=\frac{\log^{2}(n)}{n^{3}}.
\end{align*}
For each positive entire number $k$, rewrite the expansion as follows
\begin{equation}\label{eq:ActInnExp}
t_{j,n}=\om_{n}^{j} n^{(\al+1)\frac{1}{n}}\Big\{1+\sum_{s=1}^{k-1}c_{s}(\om_{n}^{j})\xi_{s}(n)+E_{k,j,n}\Big\},
\end{equation}
where $E_{k,j,n}=O(\xi_{k}(n))$ is a remainder or error term satisfying $|E_{k,j,n}|\le \ka\, \xi_{k}(n)$ for some constant $\ka$ depending only on the generating function $a$. For each fixed $j_{1}=1,\ldots,n_{1}$ let $\si_{j_{1}}\equiv\om_{n_{1}}^{j_{1}}=\cdots=\om_{n_{k-1}}^{j_{k-1}}$ (see the radially aligned red dots in Figure~\ref{fg:Grid}), and apply $k-1$ times the expansion \eqref{eq:ActInnExp}, obtaining
{\footnotesize\begin{eqnarray*}
t_{j_{1},n_{1}}&=&\si_{j_{1}}n_{1}^{(\al+1)\frac{1}{n_{1}}}\Big\{1+c_{1}(\si_{j_{1}})\xi_{1}(n_{1})+\cdots+c_{k-1}(\si_{j_{1}})\xi_{k-1}(n_{1})+E_{k,j_{1},n_{1}}\Big\},\\
t_{j_{2},n_{2}}&=&\si_{j_{1}}n_{2}^{(\al+1)\frac{1}{n_{2}}}\Big\{1+c_{1}(\si_{j_{1}})\xi_{1}(n_{2})+\cdots+c_{k-1}(\si_{j_{1}})\xi_{k-1}(n_{2})+E_{k,j_{2},n_{2}}\Big\},\\
&\vdots&\\
t_{j_{k-1},n_{k-1}}&=&\si_{j_{1}}n_{k-1}^{(\al+1)\frac{1}{n_{k-1}}}\Big\{1+c_{1}(\si_{j_{1}})\xi_{1}(n_{k-1})+\cdots+c_{k-1}(\si_{j_{1}})\xi_{k-1}(n_{k-1})+E_{k,j_{k-1},n_{k-1}}\Big\}.
\end{eqnarray*}}
Let $\hat c_{s}(\si_{j_{1}})$ be the approximation of $c_{s}(\si_{j_{1}})$ obtained by removing all the remainder terms $E_{k,j_{\ell},n_{\ell}}$ and solving the resulting linear system:
{\footnotesize\begin{equation*}
\begin{bmatrix}
\xi_{1}(n_{1}) & \xi_{2}(n_{1}) & \cdots & \xi_{k-1}(n_{1})\\[1ex]
\xi_{1}(n_{2}) & \xi_{2}(n_{2}) & \cdots & \xi_{k-1}(n_{2})\\[1ex]
\vdots & \vdots & \ddots & \vdots\\[1ex]
\xi_{1}(n_{k-1}) & \xi_{2}(n_{k-1}) & \cdots & \xi_{k-1}(n_{k-1})
\end{bmatrix}
\begin{bmatrix}
\hat c_{1}(\si_{j_{1}}) \\[1ex]
\hat c_{2}(\si_{j_{1}}) \\[1ex]
\vdots\\[1ex]
\hat c_{k-1}(\si_{j_{1}})
\end{bmatrix}
=
\si_{j_{1}}^{-1}\begin{bmatrix}
n_{1}^{-(\al+1)\frac{1}{n_{1}}}t_{j_{1},n_{1}} \\
n_{2}^{-(\al+1)\frac{1}{n_{2}}}t_{j_{2},n_{2}} \\
\vdots\\
n_{k-1}^{-(\al+1)\frac{1}{n_{k-1}}}t_{j_{k-1},n_{k-1}}
\end{bmatrix}
\!\!-\!\!
\begin{bmatrix}
1\\[1ex] 1\\[1ex] \vdots\\[1ex] 1
\end{bmatrix}.
\end{equation*}}
Thanks to properties (i)--(vi), the function $a$ is bijective, and hence the value of each $t_{j_{\ell},n_{\ell}}$ can be calculated as $a^{-1}(\la_{j_{\ell}}(T_{n_{\ell}}(a)))$. In general, the latter computation can be done numerically with any standard root finder (i.e. \texttt{FindRoot} in \textsc{Mathematica} or \texttt{fzero} in \textsc{Matlab}). As mentioned in previous works, a variant of this extrapolation strategy was first suggested by Albrecht Böttcher in \cite[\S7]{BoBo15a} and is analogous to the Richardson extrapolation employed in the context of Romberg integration \cite[\S3.4]{StBu10}.
\bigskip

\noindent\textbf{Interpolation.} The objective of this phase is to estimate $\hat c_{s}(\om_{n}^{j})$, for any index $j\in\{1,\ldots,n\}$, any $n$, and any level $s=1,\ldots,k-1$. If $\om_{n}^{j}$ coincides with one of the points in the grid $\{\om_{n_{1}}^{1},\ldots,\om_{n_{1}}^{n_{1}}\}$, then we have the approximations $\hat c_{s}(\om_{n}^{j})$ from the extrapolation phase for free. In any other case, we will do it by interpolating the data
\begin{equation*}
(\om_{n_{1}}^{1},\hat c_{s}(\om_{n_{1}}^{1})),(\om_{n_{1}}^{2},\hat c_{s}(\om_{n_{1}}^{2})),\ldots,(\om_{n_{1}}^{n_{1}},\hat c_{s}(\om_{n_{1}}^{n_{1}})),
\end{equation*}
and then evaluating the resulting polynomial at $\om_{n}^{j}$. This interpolation can be done in many ways, but to avoid spurious oscillations explained by the Runge phenomenon \cite[p.\,78]{Da75a}, and following the strategy of the previous works, we decided to do it considering only the $k-s+8$ nodes in the grid $\{\om_{n_{1}}^{1},\ldots,\om_{n_{1}}^{n_{1}}\}$ which are closest to $\om_{n}^{j}$. Those nodes can be determined uniquely unless $\om_{n}^{j}$ is the mid point of two consecutive nodes in the grid, in which case we can take any of the two possible choices.

Finally, our eigenvalue approximation with $k$ terms, is given by
\begin{equation}\label{eq:InnNAS}
\la_{j}^{\textsc{inn}(k)}(T_{n}(a))\equiv a\Big(\om_{n}^{j}n^{(\al+1)\frac{1}{n}}\Big\{1+\sum_{s=1}^{k-1}\hat c_{s}(\om_{n}^{j})\xi_{s}(n)\Big\}\Big).
\end{equation}

\begin{remark}
The approximation \eqref{eq:InnNAS} can be tuned in a number of ways. To guarantee acceptable overall accuracy, the precomputing phase must be done with a significant amount of precision digits, let us say 60. In the precomputing and extrapolation phases, we can manage more levels $k$ than in the interpolation phase, for example 2 or 3 additional levels will improve the final result. Finally, in the interpolation phase we can increase the number of interpolated points.
\end{remark}

\section{Extreme eigenvalues}\label{sc:extreme eigs}

In this section we develop a numerical algorithm to approximate the extreme eigenvalues of $T_{n}(a)$. In the paper \cite{BoGr18} the authors considered the generating function
\[a(t)=\frac{1}{t}(1-t)^{\al}f(t),\]
satisfying the following properties:
\begin{enumerate}[(i)]
\item The function $f$ belongs to $H^{\nf}$ with $f(0)\ne0$ and for some $\eps>0$, $f$ has an analytic continuation to the region
\[K_{\eps}\equiv B(1,\eps)\setminus\{x\in\bR\colon 1<x<1+\eps\},\]
and is continuous in $\hat K_{\eps}\equiv \ol{B(1,\eps)}\setminus\{x\in\bR\colon 1<x<1+\eps\}$. Additionally, $f_{\ph}(x)\equiv f(1+x\e^{\i\ph})$ belongs to the algebra $C^{2}[0,\eps]$ for each $\ph\in(-\pi,\pi]$. Here $B(z_{0},r)$ is the open disk centered at $z_{0}$ with radius $r$.
\item Let $0<\al<1$ be a constant and take the argument $-\al\pi<\arg(1-z)^{\al}\le\al\pi$ when $-\pi<\arg(1-z)\le\pi$.
\item The range $\cR(a)$ is a Jordan curve in $\bC$ with $\wind_{\la}(a)=-1$ for each $\la\in\cD(a)$.
\end{enumerate}
For the eigenvalues $\la_{j}(T_{n}(a))$ located in $W_{0}$, it was proved that
\begin{equation}\label{eq:ExtTh}
\la_{j}(T_{n}(a))=\frac{\La_{j}^{\al}}{(n+1)^{\al}}\{1+\De(\La_{j},n)\},
\end{equation}
where $\De(\La_{j},n)=O(n^{-\al})$, as $n\to\nf$, represents the remainder or error term, which satisfies the inequality $|\De(\La_{j},n)|\le\ka n^{-\al}$ for some constant $\ka$ depending only on $a$. The numbers $\La_{j}$ are certain roots of the function
\[F(\La)\equiv\frac{2\pi\i}{\al}\e^{\La}-\int_{C}\e^{-\La u}\be(u)\dif u,\]
where $C$ is the radial segment in $\bC$ with argument $\pm\frac{3}{4}\pi$ (the sign depending on the sign of the imaginary part of $\la_{j}(T_{n}(a))$) and $\be(u)\equiv(u^{\al}\e^{\i\al\pi}-1)^{-1}-(u^{\al}\e^{-\i\al\pi}-1)^{-1}$.

In our paper, instead of the exact calculation of the zeros $\La_{j}$, we use \eqref{eq:ExtTh} as an inspiration to formulate the following conjecture.

\begin{conjecture}\label{cj:Extreme}
Let $\eps$ be a small positive number, and $a$ be a generating function satisfying the previous properties $\textup{(i)}$--$\textup{(iii)}$. There exists an entire number $k\ge1$ such that the eigenvalues $\la_{j}(T_{n}(a))$ with $j=1,\ldots,\lfloor\eps n\rfloor$ or $j=n-\lfloor\eps n\rfloor+1,\ldots,n$, satisfy the expansion
\[\la_{j}(T_{n}(a))=\sum_{\ell=0}^{k-1} \frac{q_{\ell}(j)}{(n+1)^{\al+\ell}}+E_{k,j,n},\]
where
\begin{itemize}
\item $E_{k,j,n}$ is the remainder term which satisfies the inequality $|E_{k,j,n}|\le cn^{-\al-k}$ for some constant $c$ depending only on $\al$;
\item the coefficients $q_{\ell}$ $(\ell\ge1)$ are continuous functions from $\bT$ to $\bC$ depending only on $a$.
\end{itemize} 
\end{conjecture}

\subsection{Algorithm}

For notation simplicity, take $h=\frac{1}{n+1}$. Since Conjecture~\ref{cj:Extreme} does not include a factor of the type $\om_{j}^{n}$ or $\pi jh$, we cannot handle the grid trick previously used, hence this algorithm is essentially different from the one in Subsection~\ref{ssc:InnAlg} and from the algorithms employed in previous works. 

\noindent{\textbf{Precomputing.}
This phase is similar to the precomputing phase in Subsection~\ref{ssc:InnAlg}, but we can work with a different number of levels $k$. However, in order to avoid a heavy notation, we will use the same letter anyway. In this phase we calculate the eigenvalues of $T_{n}(a)$ for $n=n_{1},\ldots,n_{k-1}$.
\bigskip

\noindent\textbf{Extrapolation.} For some fixed index $j_{0}\le n_{1}$ and each $j=1,\ldots,j_{0}$ or $j=n-j_{0}+1,\ldots,n$, apply $k$ times the expansion in Conjecture~\ref{cj:Extreme}, reaching
\begin{eqnarray*}
\la_{j}(T_{n_{1}}(a))&=&q_{1}(j)h_{1}^{\al}+q_{2}(j)h_{1}^{\al+1}+\cdots+q_{k}(j)h_{1}^{\al+k-1}+E_{k,j,n_{1}},\\
\la_{j}(T_{n_{2}}(a))&=&q_{1}(j)h_{2}^{\al}+q_{2}(j)h_{2}^{\al+1}+\cdots+q_{k}(j)h_{2}^{\al+k-1}+E_{k,j,n_{2}},\\
&\vdots&\\
\la_{j}(T_{n_{k}}(a))&=&q_{1}(j)h_{k}^{\al}+q_{2}(j)h_{k}^{\al+1}+\cdots+q_{k}(j)h_{k}^{\al+k-1}+E_{k,j,n_{k}},
\end{eqnarray*}
where $E_{k,j,n}=O(h^{\al+k})$ is a remainder term satisfying the inequality $|E_{k,j,n}|\le c\,h^{\al+k}$ for some constant $c$ depending only on $a$. As in the inner eigenvalue case, let $\hat q_{\ell}(j)$ be the approximation of $q_{\ell}(j)$ obtained by removing all the remainder terms $E_{k,j,n_{\ell}}$ and solving the resulting linear system:
\begin{equation}\label{eq:ExtExtPhase}
\begin{bmatrix}
h_{1}^{\al} & h_{1}^{\al+1} & \cdots & h_{1}^{\al+k-1}\\
h_{2}^{\al} & h_{2}^{\al+1} & \cdots & h_{2}^{\al+k-1}\\
\vdots & \vdots & \ddots & \vdots\\
h_{k}^{\al} & h_{k}^{\al+1} & \cdots & h_{k}^{\al+k-1}
\end{bmatrix}
\begin{bmatrix}
\hat q_{1}(j)\\
\hat q_{2}(j)\\
\vdots\\
\hat q_{k}(j)
\end{bmatrix}
=
\begin{bmatrix}
\la_{j}(T_{n_{1}}(a))\\
\la_{j}(T_{n_{2}}(a))\\
\vdots\\
\la_{j}(T_{n_{k}}(a))
\end{bmatrix}.
\end{equation}
Finally, our extreme eigenvalue approximation with $k$ terms is given by
\begin{equation}\label{eq:ExtNAS}
\la_{j}^{\textsc{ext}(k)}(T_{n}(a))\equiv \sum_{\ell=0}^{k-1} \frac{\hat q_{\ell}(j)}{(n+1)^{\al+\ell}}.
\end{equation}

\section{Numerical experiments}\label{sc:num}

In this section we test the accuracy of the proposed algorithms related to the formula (\ref{eq:InnNAS}) and to the formula \eqref{eq:ExtNAS}, for approximating the inner and extreme eigenvalues of $T_{n}(a)$, respectively. Our numerical experiments are performed by using \textsc{Mathematica v.12 (64~bit)} on a platform with 16GB RAM, using an Intel processor QuadCore IntelCore i7 2.6 GHz. For a generating function $a$ and a level $k$, consider the individual absolute errors
\[\textrm{AE}_{j,n}^{\textsc{inn}(k)}\equiv |\la_{j}(T_{n}(a))-\la_{j}^{\textsc{inn}(k)}(T_{n}(a))|,\quad
\textrm{AE}_{j,n}^{\textsc{ext}(k)}\equiv |\la_{j}(T_{n}(a))-\la_{j}^{\textsc{ext}(k)}(T_{n}(a))|.\]
To improve the readability of the obtained errors, we decide to define the maximum error
\[\textrm{AE}_{n}^{\textsc{inn}(k)}\equiv\max\{\textrm{AE}_{j,n}^{\textsc{inn}(k)}\colon j=\lfloor\eps n\rfloor,\ldots,n-\lfloor\eps n\rfloor\},\]
and the individual relative errors
\[\textrm{RE}_{j,n}^{\textsc{ext}(k)}\equiv\frac{\textrm{AE}_{j,n}^{\textsc{ext}(k)}}{|\la_{j}(T_{n}(a))|},\qquad
\textrm{RE}_{j,n}^{\textsc{inn}(k)}\equiv\frac{\textrm{AE}_{j,n}^{\textsc{inn}(k)}}{|\la_{j}(T_{n}(a))|}.\]
Given an $\eps>0$, our global individual relative error is
\[\textrm{RE}_{j,n}^{\eps}\equiv\begin{cases}
\textrm{RE}_{j,n}^{\textsc{ext}(k_{2})} &\quad\mbox{if}\quad j\in[1,\lfloor\eps n\rfloor]\cup[n-\lfloor\eps n\rfloor+1,n];\\
\textrm{RE}_{j,n}^{\textsc{inn}(k_{1})}&\quad\mbox{if}\quad j\in[\lfloor\eps n\rfloor+1,n-\lfloor\eps n\rfloor],
\end{cases}\]
where $k_{1}$ and $k_{2}$ are the largest available. Note that these constants are limited by the precomputing phase, and that $\eps$ must be selected within the range $0<\eps<\frac{n_{1}}{n}$.

For $\al\in\bR_{+}\setminus\bZ$, consider the generating function
\[a(t)=\frac{1}{t}(1-t)^{\al}.\]
A detailed description of this function and all the related neighborhoods appears in Figure~\ref{fg:Neigh}. Using the binomial theorem we obtain
\[a(t)=\sum_{j=-1}^{\nf}(-1)^{j+1}\binom{\al}{j+1} t^{j},\]
then the Fourier coefficients $a_{j}$ are given by $a_{j}=(-1)^{j+1}\binom{\al}{j+1}$ for $j\ge-1$ and $a_{j}=0$ in any other case. It is also worth noticing that since the matrices $T_{n}(a)$ are all real, then the eigenvalues appear as conjugate pairs and this explains the symmetry of the error in the related figures.

\setlength{\tabcolsep}{2pt}
\begin{table}[H]
\centering
{\footnotesize\begin{tabular}{rllllll}
\toprule
\multicolumn{1}{c}{$n$} & \multicolumn{1}{c}{$256$} & \multicolumn{1}{c}{$512$} & \multicolumn{1}{c}{$1024$} & \multicolumn{1}{c}{$2048$} & \multicolumn{1}{c}{$4096$} & \multicolumn{1}{c}{$8192$}\\ \midrule

$\textrm{AE}_{n}^{\textsc{inn}(1)}$ & $1.5783\x10^{-2}$ & $8.0820\x10^{-3}$ & $4.0956\x10^{-3}$ & $2.0631\x10^{-3}$ & $1.0358\x10^{-3}$ & $5.1906\x10^{-4}$ \\
$\xi_{1}(n)\textrm{AE}_{n}^{\textsc{inn}(1)}$ & $4.0405\x10^{0}$ & $4.1380\x10^{0}$ & $4.1938\x10^{0}$ & $4.2252\x10^{0}$ & $4.2426\x10^{0}$ & $4.2521\x10^{0}$\\ \midrule

$\textrm{AE}_{n}^{\textsc{inn}(2)}$ & $3.0955\x10^{-4}$ & $8.8936\x10^{-5}$ & $2.5152\x10^{-5}$ & $7.0464\x10^{-6}$ & $1.9630\x10^{-6}$ & $5.4512\x10^{-7}$\\
$\xi_{2}(n)\textrm{AE}_{n}^{\textsc{inn}(2)}$ & $5.0716\x10^{0}$ & $4.9012\x10^{0}$ & $4.6623\x10^{0}$ & $4.3934\x10^{0}$ & $4.1167\x10^{0}$ & $3.8452\x10^{0}$\\ \midrule

$\textrm{AE}_{n}^{\textsc{inn}(3)}$ & $2.3674\x10^{-4}$ & $6.6709\x10^{-5}$ & $1.8445\x10^{-5}$ & $5.0358\x10^{-6}$ & $1.3625\x10^{-6}$ & $3.6613\x10^{-7}$\\
$\xi_{3}(n)\textrm{AE}_{n}^{\textsc{inn}(3)}$ & $2.7980\x10^{0}$ & $2.8032\x10^{0}$ & $2.7904\x10^{0}$ & $2.7702\x10^{0}$ & $2.7482\x10^{0}$ & $2.7268\x10^{0}$\\ \midrule

$\textrm{AE}_{n}^{\textsc{inn}(4)}$ & $7.8781\x10^{-5}$ & $2.0791\x10^{-5}$ & $5.3823\x10^{-6}$ & $1.3770\x10^{-6}$ & $3.4968\x10^{-7}$ & $8.8421\x10^{-8}$\\
$\xi_{4}(n)\textrm{AE}_{n}^{\textsc{inn}(4)}$ & $5.1630\x10^{0}$ & $5.4503\x10^{0}$ & $5.6438\x10^{0}$ & $5.7754\x10^{0}$ & $5.8666\x10^{0}$ & $5.9338\x10^{0}$\\ \midrule

$\textrm{AE}_{n}^{\textsc{inn}(5)}$ & $1.1053\x10^{-5}$ & $2.0675\x10^{-6}$ & $3.8406\x10^{-7}$ & $7.1041\x10^{-8}$ & $1.3076\x10^{-8}$ & $2.3356\x10^{-9}$\\
$\xi_{5}(n)\textrm{AE}_{n}^{\textsc{inn}(5)}$ & $1.1589\x10^{1}$ & $1.2264\x10^{1}$ & $1.2887\x10^{1}$ & $1.3484\x10^{1}$ & $1.4040\x10^{1}$ & $1.4187\x10^{1}$\\ \midrule

$\textrm{AE}_{n}^{\textsc{inn}(6)}$ & $8.2836\x10^{-6}$ & $1.4196\x10^{-6}$ & $2.3767\x10^{-7}$ & $3.9140\x10^{-8}$ & $6.3845\x10^{-9}$ & $1.0800\x10^{-9}$ \\
$\xi_{6}(n)\textrm{AE}_{n}^{\textsc{inn}(6)}$ & $6.2656\x10^{0}$ & $6.4210\x10^{0}$ & $6.5083\x10^{0}$ & $6.5549\x10^{0}$ & $6.5934\x10^{0}$ & $6.9258\x10^{0}$\\ \bottomrule
\end{tabular}}
\caption{The maximum errors $\textup{AE}_{n}^{\textsc{inn}(k)}$, and maximum normalized errors $\xi_{k}(n)\textup{AE}_{n}^{\textsc{inn}(k)}$ for the levels $k=1,\ldots,6$, indexes $j\in\big(\frac{n}{8},\frac{7n}{8}\big)$, and different matrix sizes $n$, corresponding to the matrix $T_{n}(a)$ with $a(t)=\frac{1}{t}(1-t)^{\frac{3}{4}}$. We used a grid of size $n_{1}=100$.}\label{tb:InnErru}
\end{table}

Figures~\ref{fg:ErruExt} and \ref{fg:ErruInn} show a comparison between the individual relative errors $\textrm{RE}_{j,n}^{\textsc{ext}(k)}$ and $\textrm{RE}_{j,n}^{\textsc{inn}(k)}$, obtained by our expansions for different number of terms $k$. We can see that our algorithms are producing smaller errors as $k$ increases, and that the overall reached precision is very high.

Table~\ref{tb:InnErru} shows the maximum individual and normalized inner eigenvalue errors. We considered only the indexes $j\in\big(\frac{n}{8},\frac{7n}{8}\big)$, in this way we are excluding the eigenvalues belonging to the neighborhood $W_{0}$ corresponding to $\eps=\frac{1}{8}$. The almost constant behavior of the normalized errors is a verification of the expansion in Theorem~\ref{th:InnerExp}. We remark that the error $\textup{AE}^{\textsc{inn}(1)}_{n}$ has order $O(\frac{1}{n})$ which is the same that the distance between two consecutive eigenvalues have, which means that the approximation with one term can be only used to understand the distribution of the eigenvalues.

\begin{figure}[H]
\centering
\includegraphics[width=0.8\textwidth]{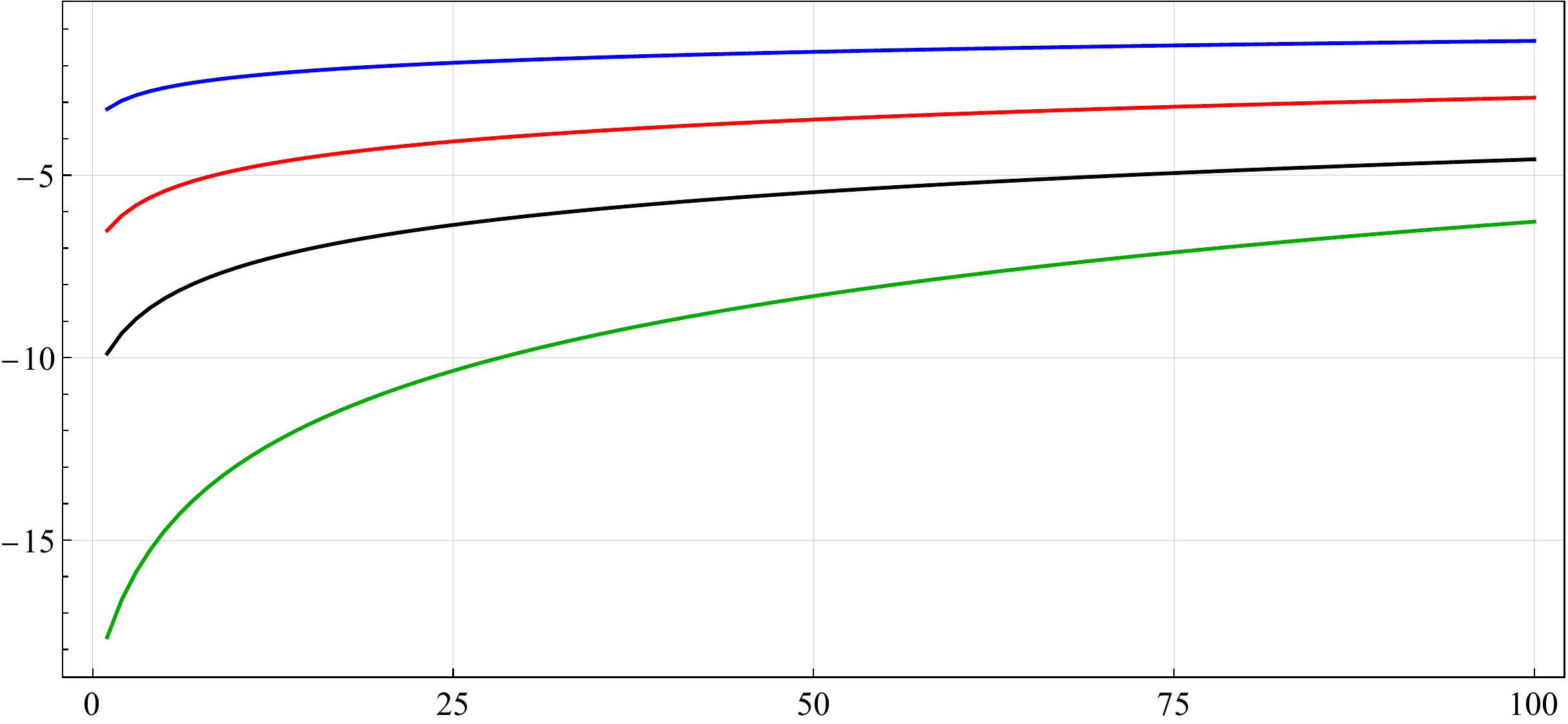}
\caption{The base-10 logarithm extreme eigenvalue relative errors $\textup{RE}^{\textsc{ext}(k)}_{j,n}$ $(j=1,\ldots,100)$ for the matrix $T_{n}(a)$ with $a(t)=\frac{1}{t}(1-t)^{\frac{3}{4}}$, a grid size $n_{1}=100$, a matrix size $n=8192$, and different number of terms $k=1$ (blue), $k=2$ (red), $k=3$ (black), and $k=6$ (green). All the errors are plotted over the grid $\big\{\frac{2\pi j}{n+1}\big\}_{j=1}^{n}$.}\label{fg:ErruExt}
\end{figure}
\begin{figure}[H]
\centering
\includegraphics[width=0.8\textwidth]{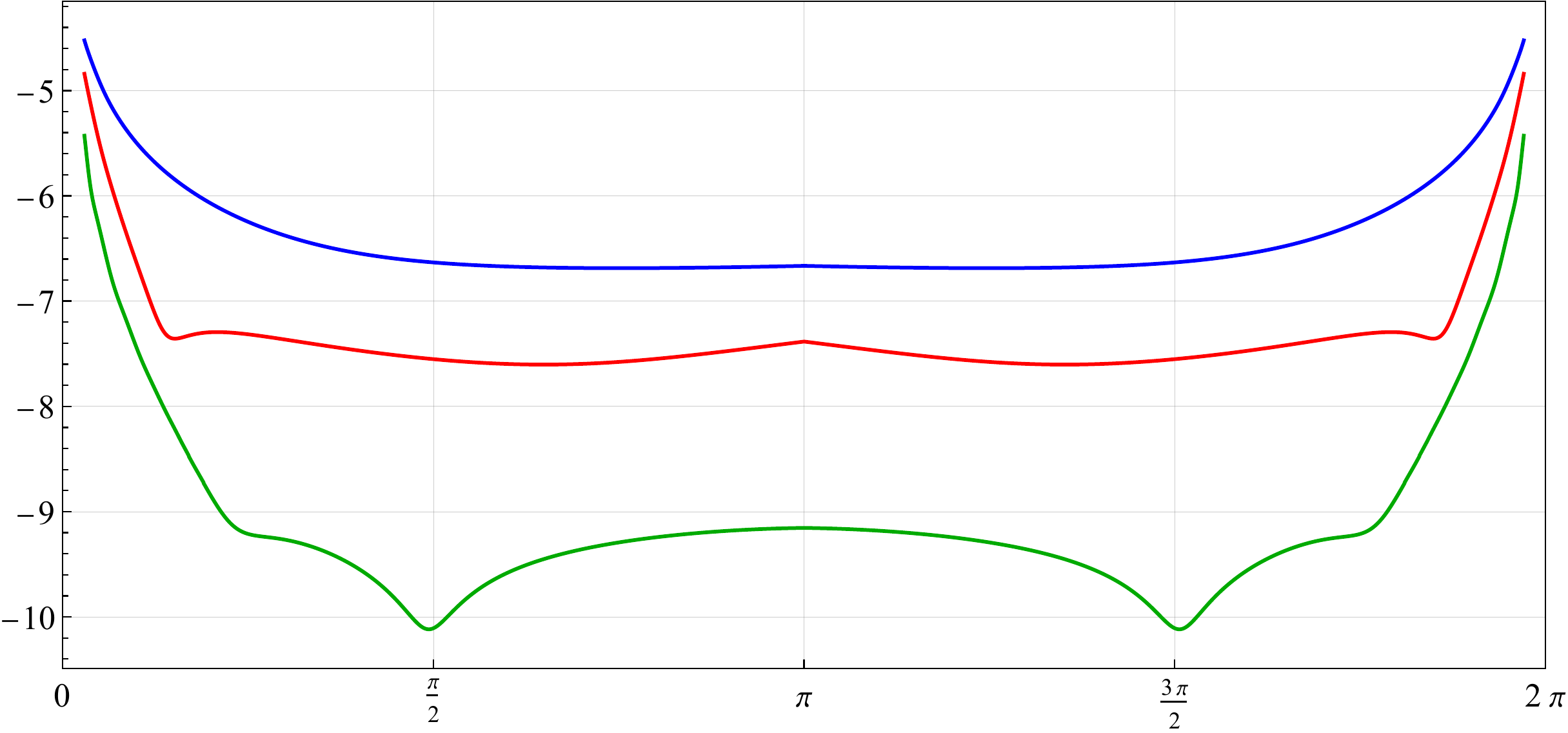}
\caption{The same as Figure~\ref{fg:ErruExt} for the inner eigenvalue relative errors $\textup{RE}_{j,n}^{\textsc{inn}(k)}$ with $\eps=\frac{1}{64}$, and for $k=3$ (blue), $k=4$ (red), and $k=5$ (green). All the errors are plotted over the grid $\big\{\frac{2\pi j}{n+1}\big\}_{j=1}^{n}$.}\label{fg:ErruInn}
\end{figure}
\begin{figure}[H]
\centering
\includegraphics[width=0.8\textwidth]{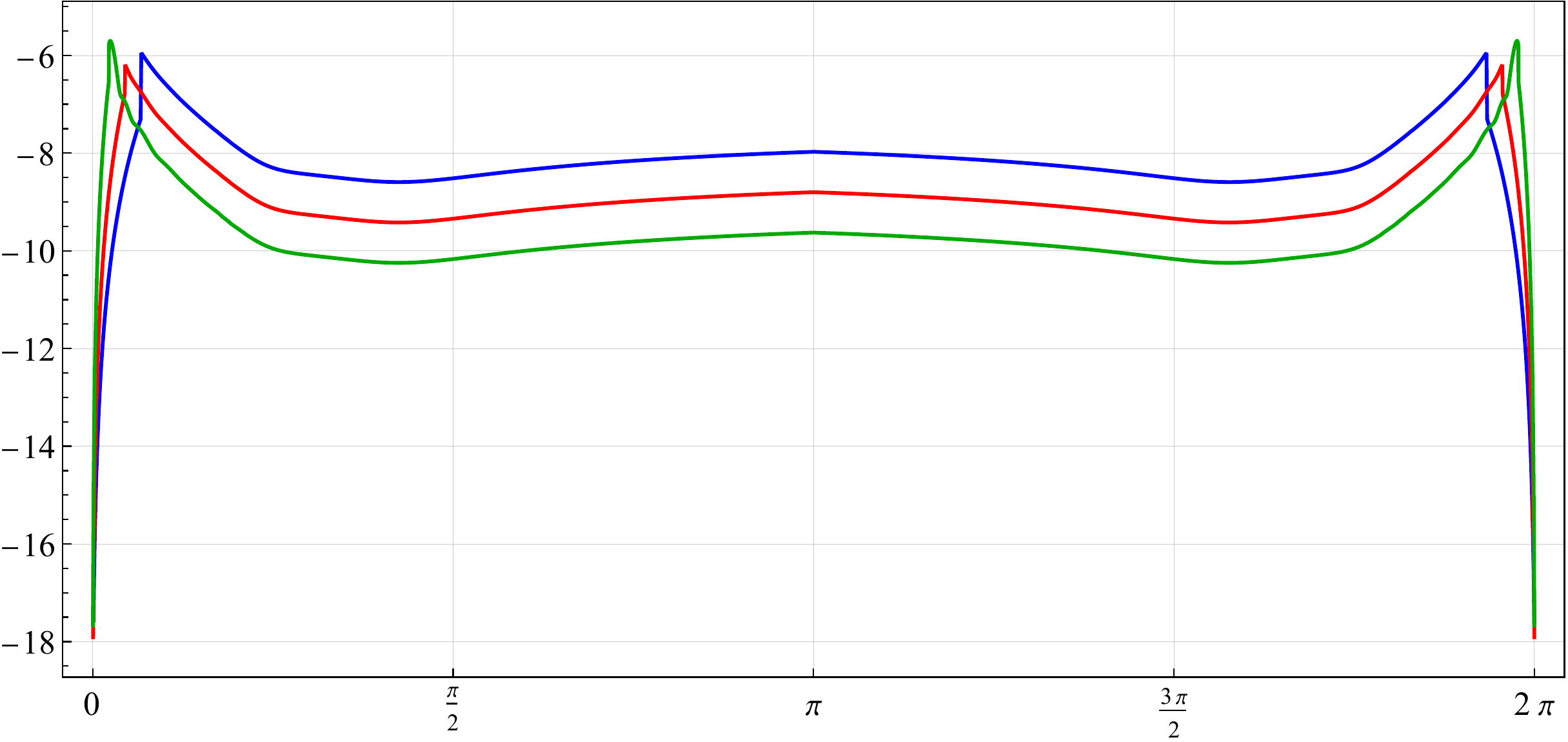}
\caption{The base-10 logarithm for the global individual relative errors $\textup{RE}_{j,n}^{\eps}$ for the matrix $T_{n}(a)$ with $a(t)=\frac{1}{t}(1-t)^{\frac{3}{4}}$, a grid size $n_{1}=100$, levels $k_{1}=7$, $k_{2}=7$, and different matrix sizes and epsilons: $n=2048$, $\eps=\frac{1}{30}$ (blue), $n=4096$, $\eps=\frac{1}{45}$ (red), and $n=8192$, $\eps=\frac{1}{90}$ (green). All the errors are plotted over the grid $\big\{\frac{2\pi j}{n+1}\big\}_{j=1}^{n}$.}\label{fg:E1InnErr}
\end{figure}

Figure~\ref{fg:E1InnErr} shows the global individual relative errors $\textrm{RE}_{j,n}^{\eps}$ for different matrix sizes $n$ and epsilons. We can see the high precision attained by our global algorithm, specially at the extreme eigenvalues.

\section{Future work}\label{sc:end}

The case $m<\al<m+1$ for $m=1,2,\ldots$ is very interesting for the applications. We will consider those cases in a future project.

It would be of interest to use such results also for the fast solution of the resulting large linear systems, so covering the preconditioned case and the setting of FDEs defined on $d$-dimensional domains with $d\ge 2$.

\section*{Acknowledgments}

Part of the numerical computations were carried in the computer center Jürgen Tischer of the mathematics department at Universidad del Valle.

The research of S.M. Grudsky was supported by CONACYT (Mexico) project “Ciencia de Frontera”
FORDECYT-PRONACES/61517/2020 and by Regional Mathematical Center of the Southern Federal University with the support of the Ministry of Science and Higher Education of Russia, Agreement 075-02-2021-1386.

The research of the third author is partly supported by the Italian INdAM-GNCS agency.

\bibliographystyle{acm}
\bibliography{Toeplitz}
%\bibliography{/Users/mbogoya/Documents/Research/MainMatters/Toeplitz}
\end{document}